\documentclass{amsart}
\usepackage{amsfonts}
\usepackage{amssymb}
\newtheorem{theorem}{Theorem}
\newtheorem{lemma}{Lemma}
\newtheorem{corollary}{Corollary}
  \newtheorem{definition}{Definition}
 \newtheorem{proposition}{Proposition}
 
 \usepackage{xcolor}

\newcommand{\black}{\color{black}}

\title{ Star-Varieties of proper central exponent greater than two }

\author{Francesca S. Benanti}
\address{Dipartimento di Matematica e Informatica, Universit\`a di Palermo, Via Archirafi 34, 90123, Palermo, Italy}
\email{francescasaviella.benanti@unipa.it}

\author{Angela Valenti}
\address{Dipartimento di Ingegneria, Universit\`a di Palermo, Viale delle Scienze, 90128, Palermo, Italy}
\email{angela.valenti@unipa.it}

\thanks{The authors are partially supported by  Fondo Finalizzato alla Ricerca dell'Università degli Studi di Palermo  (FFR) – Anno 2025.}

\subjclass[2020]{Primary 16R10, 16R50; Secondary 16P90, 16W10.}

\keywords{algebras with involution, central polynomial, polynomial identity, growth.}

\begin{document}

\begin{abstract}
		Let $F$ be a field of characteristic zero and let $ \mathcal V^* $ be a variety of associative $F$-algebras with involution *.   Associated to   $ \mathcal V^* $  are  three sequences: the sequence of \(*\)-codimensions \( c^{*}_n(\mathcal V^*) \), the sequence of  central \(*\)-codimensions \( c^{*,z}_n(\mathcal V^*) \) and the sequence of proper central \(*\)-codimensions \( c^{*,\delta}_n(\mathcal V^*) \).
        These sequences provide information on the growth of, respectively, the *-polynomial identities, the  central *-polynomial and the proper central *-polynomial of any generating algebra with involution $A$ of $ \mathcal V^*.$

        In \cite{MR2022} it was  proved that $exp^{*,\delta}(\mathcal V^*)=\lim_{n\to\infty}\sqrt[n]{c_n^{*,\delta}(\mathcal V^*)}$ exists and is an integer called the proper central $*$-exponent.
   The aim of this paper is to study the varieties of associative
		algebras with involution of proper central $*$-exponent greater than two.
        To this end we construct a finite list of algebras with involution and we prove that if $exp^{*,\delta}(\mathcal V^*) >2$, then at least one of these algebras belongs to $\mathcal V^*$.

	\end{abstract}
\date{}

\maketitle

\section{Introduction}

 Let $ A $ be an associative algebra with involution \( * \) over a field \( F \) of characteristic zero and let us denote by $A^+ = \{a \in A \,| \,a^*= a\} $ and $A^- = \{a \in A \, | \, a^*= -a \} $ the sets of symmetric and skew elements of $A.$ Let \( F\langle X, * \rangle = F\langle x_1,x^*_1, x_2,x^*_2 \dots \rangle\)   be the free associative algebra with involution, freely generated over \( F \) by a countable set \( X \) of non-commutative variables. We shall denote by $x_i^+= x_i +x^*_i$ the symmetric variables and by $x_i^-= x_i -x^*_i$ the skew variables.
    A $*${-polynomial}
\[
f(x^+_1, \dots, x^+_r, x^-_1, \dots, x^-_s) \in F\langle X, * \rangle
\]
is said to be a {central $*$-polynomial} for $A$ if, for all substitutions
$a^+_1, \dots, a^+_r \in A^+$ and $a^-_1, \dots, a^-_s \in A^-$, we have
\[
f(a^+_1, \dots, a^+_r, a^-_1, \dots, a^-_s) \in Z(A),
\]
where $Z(A)$ denotes the center of $A$.

    If \( f \) always evaluates to zero, it is called a $*$-{polynomial identity} of \( A \). If \( f \) evaluates to a nonzero element in \( Z(A) \), then it is said to be a proper central $*$-polynomial.

  Let \( P^*_n \) denote the space of multilinear \(*\)-polynomials in \( n \) variables.
Associated with this space we define the following three numerical sequences:
  \( c^*_n(A) \),  the dimension of \( P^*_n \) modulo the $*$-identities of \( A \),
  	 \( c^{*,z}_n(A) \), the dimension of \( P^*_n \) modulo the central $*$-polynomials of \( A \),
  	 and
 $$ c^{*,\delta}_n(A) = c^*_n(A) - c^{*,z}_n(A),
 $$
  that is the dimension of the space of proper central \(*\)-polynomials of degree \( n \).

 If \( A \) satisfies a non-trivial polynomial identity (i.e., \( A \) is a PI-algebra), the sequence of \(*\)-codimensions \( c^*_n(A) \) (for \( n = 1, 2, \ldots \)) is exponentially bounded, as established in \cite{Amitsur}, \cite{GR1985} and \cite{BGZ1999}. The same holds for the sequences of central \(*\)-codimensions \( c^{*,z}_n(A) \) and proper central \(*\)-codimensions \( c^{*,\delta}_n(A) \).

  Computing the explicit values of these sequences is extremely difficult, with only a few known examples. Therefore the research has instead focused on analyzing their asymptotic behavior. It turns out that the exponential growth of all three sequences has been studied and the following limits,  called the $\ast$-exponent, the central  $\ast$-exponent and the proper central $\ast$-exponent,  exist and are integers \cite{GPV2017}, \cite{MR2022}:
    \[
    exp^\ast(A) = \lim_{n \to \infty} \sqrt[n]{c^*_n(A)}, \quad
    exp^{\ast,z}(A) =\lim_{n \to \infty} \sqrt[n]{c^{*,z}_n(A)}, \quad
    exp^{\ast,\delta}(A) =\lim_{n \to \infty} \sqrt[n]{c^{*,\delta}_n(A)}.
    \]

    These three limits are invariants of \( \operatorname{Id}^*(A)  \),  the  $T^*$-ideal  of the *-polynomial identities of $A$, and it is rather remarkable that they correspond to the dimensions of certain subalgebras of a finite-dimensional superalgebra closely associated with this $T^*$-ideal.

   A crucial feature of these sequences is that they either grow exponentially or are polynomially bounded \cite{GPV2017}, \cite{MR2022}. Hence intermediate growth does not occur.

   These ideas were first introduced in the general framework of associative algebras, without additional structure. For an overview of the theory of PI-algebras, we refer the reader to \cite{{GZbook}} and \cite{{AGPR}}. In this context, analogous sequences are defined: $c_n(A),$ ${c^{z}_n(A)}$ and $c^{\delta}_{n}(A)$, representing the ordinary codimensions, central codimensions and proper central codimensions, respectively. For a   PI-algebra $A$, all three sequences exhibit exponential growth and their growth rates have been computed in several key papers (see \cite{GZ1}, \cite{GZ2}, \cite{GZ2018}, \cite{GZ2019}).
    It turns out that if $A$ satisfies a non-trivial identity, then $\lim_{n\to \infty} \sqrt[n]{c_n(A)}=exp(A)$, $\lim_{n\to \infty}  \sqrt[n]{c_n^z(A)} =exp^z(A)$
	and $\lim_{n\to \infty}  \sqrt[n]{c_n^\delta(A)} =exp^\delta(A)$ exist,
	are integers  called respectively the PI-exponent, the central exponent and the proper central exponent of $A$. Moreover all three codimensions either grow exponentially or
	are polynomially bounded.

   In the language of varieties, it's well established that the set of the $*$-polynomial identities satisfied by an algebra $A$ with involution generate a $*$-variety $\mathcal{V^*} = \text{var}(A)$ and the growth of  the $*$-codimensions of $A$ reflects the growth of the variety itself. Similarly, the $z$-growth and the $\delta$-growth of the $*$-variety $\mathcal{V^*}$ correspond to the growth of  ${c^{*,z}_n(A)}$ and  ${c^{*,\delta}_n(A)}$, respectively.

In the ordinary case of associative algebras, Giambruno and Zaicev in \cite{GZ2000} characterized the varieties of exponent greater than two.
Recently, the varieties of associative algebras and the varieties of $G$-graded algebras with proper central exponent greater than two were investigated in \cite{BV2025}, \cite{BV20252}.

In the setting of algebras with involution,
it was shown in \cite{GM2001},  \cite{GM2001comm} and \cite{MV2000} that only two $*$-varieties of $*$-algebras exhibit almost polynomial growth.
More recently
in \cite{GLP2026} the authors  classified  $*$-varieties having almost polynomial $z$-growth  or $\delta$-growth. Recall that a $*$-variety $\mathcal{V^\ast}$  is said to have almost polynomial growth (or $z$-growth or $\delta$-growth) if it exhibits exponential growth but every proper subvariety grows only polynomially. Moreover it is
minimal of $*$-exponent $d$ if $exp^\ast(A) = d$ and  the $\ast$-exponent of any proper $*$-subvariety is strictly less than
$d$.
Analogously, one defines minimal $*$-varieties of central $*$-exponent
and of proper central $*$-exponent~$d$.

The aim of this paper is to study the $*$-varieties of algebras with involution whose  proper central $*$-exponent is greater than two.  We  explicitly exhibit a family  of $*$-algebras $A_i$ in order to prove that, if a $*$-variety $\mathcal V^*$ has
proper central $*$-exponent greater than two, then $A_i \in \mathcal V^*,$ for some $i.$ Moreover we classify the minimal $*$-varieties of algebras with involution whose  proper central $*$-exponent is equal to three or four.

\section*{Preliminaries and basic results}

Throughout we shall assume that $F$ is an algebraically closed field of characteristic zero. Let $A$ be an associative $F$-algebra endowed with an involution $\,\ast\,$, that is a linear map $\ast : A \rightarrow A$ of order at most two such that
$
(ab)^\ast = b^\ast a^\ast, \, \text{for all } a, b \in A.
$
We denote the set of symmetric and skew-symmetric elements of $A$ by
$
A^+ = \{ a \in A \mid a^\ast = a \}$ and $A^- = \{ a \in A \mid a^\ast = -a \}
$
respectively, so that $A = A^+ \oplus A^-$.

Let ${F} \langle X, \ast \rangle = {F} \langle x_1, x_1^\ast, x_2, x_2^\ast, \ldots \rangle$ denote the free associative algebra with involution generated by the countable set $X = \{x_1, x_2, \ldots\}$. Define the symmetric and skew-symmetric variables
$x_i^+ = x_i + x_i^\ast  \text{and } x_i^- = x_i - x_i^\ast, \, \text{for all } i \geq 1,$
hence
${F} \langle X, \ast \rangle = {F} \langle x_1^+, x_1^-, x_2^+, x_2^-, \ldots \rangle.
$

Let $\mathrm{Id}^\ast(A)$ be the  $T^\ast$-ideal of all $\ast$-polynomial identities of $A$ and $\mathrm{Id}^{\ast,z}(A)$ be the $T^\ast$-space of central $*$-polynomial of $A.$
For each $n \geq 1$, let $P_n^\ast$ be the vector space of multilinear $\ast$-polynomials in the variables $x_1, x_1^* \ldots, x_n, x_n^*.$
Define the following codimension sequences:
$$
c_n^\ast(A) = \dim  \frac{P_n^\ast}{P_n^\ast \cap \operatorname{Id}^\ast(A)}, \quad
c_{n}^{\ast,z}(A) = \dim \frac{P_n^\ast}{P_n^\ast \cap \operatorname{Id}^{\ast,z}(A)}, \quad
c_{n}^{\ast,\delta}(A) = \dim  \frac{P_n^\ast \cap \operatorname{Id}^{\ast,z}(A)}{P_n^\ast \cap \operatorname{Id}^\ast(A)} .
$$

These sequences are called the sequence of $\ast$-codimensions,  central $\ast$-codimensions  and  proper central $\ast$-codimensions of $A$, respectively.
If $A$ is a PI-algebra  these sequences grow exponentially and it is known that the limits
\[
exp^\ast(A) = \lim_{n \to \infty} \sqrt[n]{c_n^\ast(A)}, \quad
exp^{\ast,z}(A) = \lim_{n \to \infty} \sqrt[n]{c_{n}^{\ast,z}(A)}, \quad
exp^{\ast,\delta}(A) = \lim_{n \to \infty} \sqrt[n]{c_{n}^{\ast,\delta}(A)}
\]
exist and are integers, known as the {$\ast$-exponent}, the central $\ast$-exponent and the {proper central $\ast$-exponent} of $A$ (see \cite{GPV2017}, \cite{MR2022}).

A central tool in the study of $\ast$-polynomial identities in PI-algebras is the theory of superalgebras and superinvolutions.
A superinvolution on a superalgebra $B = B^{(0)} \oplus B^{(1)}$ is a graded linear map $\sharp : B \to B$ such that:
\[
(a^\sharp)^\sharp = a, \quad (ab)^\sharp = (-1)^{|a||b|} b^\sharp a^\sharp,
\]
for all homogeneous elements $a, b \in B$.

Let $E$ denote the infinite-dimensional Grassmann algebra generated by $\{e_1, e_2, \ldots\}$ with  relations $e_i e_j = - e_j e_i$. It carries a natural $\mathbb{Z}_2$-grading: $E = E^{(0)} \oplus E^{(1)}$ and admits a superinvolution defined by $e_i^\sharp = -e_i$.
Given a superalgebra $B$ with superinvolution $\star$ its {Grassmann envelope}  is defined as:
\[
E(B) = (E^{(0)} \otimes B^{(0)}) \oplus (E^{(1)} \otimes B^{(1)}).
\]
$E(B)$ becomes an algebra with involution $\ast$ defined on homogeneous elements by:
\[
(g \otimes a)^\ast = g^\sharp \otimes a^{\star}.
\]

A  famous theorem of Kemer asserts that an arbitrary PI-algebra has the same identities as the Grassmann envelope $E(B)$ of a finite dimensional superalgebra
$B.$ This result was extended in the setting of algebras with involution in the following theorem

\medskip
\noindent
\begin{theorem}(\cite[Theorem 4]{AGK}) If $A$ is an algebra with involution satisfying a non-trivial $\ast$-identity, then there exists a finite-dimensional superalgebra with superinvolution $B$ such that
$
\operatorname{Id}^\ast(A) = \operatorname{Id}^\ast(E(B)).$
\end{theorem}

\smallskip

Let us observe that if $A$ is a finite-dimensional algebra with involution, then taking it with trivial grading we have:
\[
\mathrm{Id}^\ast(E(A)) = \mathrm{Id}^\ast(E^{(0)} \otimes A) = \mathrm{Id}^\ast(A),
\]
so we may work with $A$ directly instead of its Grassmann envelope.

To compute the $\ast$-exponential growth of codimensions, one focuses on $E(B)$. Let $B$ be a finite dimensional superalgebra with superinvolution over $F$ such that
$Id^*(A) = Id^*(E(B)).$ Since $F$  is algebraically closed,  by the Wedderburn–Malcev decomposition
 (\cite{GIL}, Theorem 4.1), $B = \bar B +J$ where $\bar B= B_1 \oplus \ldots \oplus B_m, $ $B_i$ are simple superalgebra with superinvolution and $J=J^\ast$ is the Jacbson radical.
\smallskip

 At this point  let us  recall the classification of the finite-dimensional simple superalgebras with superinvolution over an algebraically closed field of characteristic zero.

We remember that \( A \) is a simple $\ast$-superalgebra if \( A^2 \neq 0 \) and \( A \) has no non-trivial $\ast$-superideals.

It is known (see \cite{GZbook}, \cite{Kemerbook}) that if \( F \) is algebraically closed of characteristic \( \neq 2 \), a finite dimensional simple superalgebra \( A \) over $F$ is isomorphic to one of the following algebras

 \begin{itemize}
     \item $M_{k,l}(F) = M_{k+l}(F) =(M_{k,l}(F))^{(0)} \oplus (M_{k,l}(F))^{(1)} $ with $k \ge 1, k \ge l\ge 0$ where
$$
(M_{k,l}(F))^{(0)} = \left\{ \begin{pmatrix} X & 0 \\ 0 & T \end{pmatrix} \, \middle| \,
X \in M_k(F), \, T \in M_l(F)  \right\},$$

$$(M_{k,l}(F))^{(1)} = \left\{\begin{pmatrix} 0 & Y \\ Z & 0 \end{pmatrix} \,\middle| \, Y \in M_{k \times l}, \, Z \in M_{l \times k}(F) \right\}
$$

\item
$Q(n) = M_n(F \oplus cF)$  with grading  \( Q(n)^{(0)} = M_n(F) \), \( Q(n)^{(1)} = c M_n(F) \) and \( c^2 = 1 \).

 \end{itemize}

 \bigskip

If \( A \) is a superalgebra, we denote by \( A^{\text{sop}} \)  the superalgebra with the same graded structure and multiplication
$
a \circ b = (-1)^{|a||b|} ba.
$
 Then \( R = A \oplus A^{\text{sop}} \) is a superalgebra  with
$
    R^{(0)} = A^{(0)} \oplus (A^{\text{sop}})^{(0)},  R^{(1)} = A^{(1)} \oplus (A^{\text{sop}})^{(1)}
$
and  exchange superinvolution  defined by
$
(a, b)^{\text{exc}} = (b, a).
$

\smallskip

 We have the following

\begin{theorem} (\cite{BTT}). Let \( A \) be a finite-dimensional simple superalgebra with superinvolution over an algebraically closed field \( F \) of characteristic \( \neq 2 \). Then \( A \) is isomorphic to one of:

\begin{itemize}

\item
\( M_{k,k}(F) \) with transpose superinvolution:
\[
\begin{pmatrix}
X & Y \\
Z & T
\end{pmatrix}^{\text{trp}} =
\begin{pmatrix}
T^t & -Y^t \\
Z^t & X^t
\end{pmatrix}.
\]

\item  \( M_{k,2s}(F) \) with orthosymplectic involution:
\[
\begin{pmatrix}
X & Y \\
Z & T
\end{pmatrix}^{\text{osp}} =
\begin{pmatrix}
I_k & 0 \\
0 & Q
\end{pmatrix}^{-1}
\begin{pmatrix}
X & -Y \\
Z & T
\end{pmatrix}^t
\begin{pmatrix}
I_k & 0 \\
0 & Q
\end{pmatrix},
\]
where \( Q = \begin{pmatrix} 0 & I_s \\ -I_s & 0 \end{pmatrix} \).

\item \( M_{k,l}(F) \oplus M_{k,l}(F)^{\text{sop}} \) with exchange superinvolution.

\item \( Q(n) \oplus Q(n)^{\text{sop}} \) with exchange superinvolution.

\end{itemize}
\end{theorem}

\medskip
If $A$ is a simple superalgebra with superinvolution and the grading is trivial from the above theorem
one obtains the classification of the finite-dimensional simple algebras with involution, that is, $M_k(F)$ with transpose or symplectic
 involution or $M_k(F)\oplus M_k(F)^{op}$ with exchange involution, where $M_k(F)^{op}$ is the opposite algebra of $M_k(F).$

In order to estimate the $*$-exponent and the proper central $*$-exponent  we recall  that a semisimple superalgebra with superinvolution $C = B_{i_1} \oplus \cdots \oplus B_{i_k} \subseteq \overline{B}$,  where $i_1, \ldots,i_k \in
\{1, . . .,m \}$ are distinct, is {admissible} ifIn
$
B_{i_1} J \cdots J B_{i_k} \neq 0.$
We call $C$ centrally admissible for $E(B)$ if there exists a proper central $\ast$-polynomial $f(x_1^+, \ldots, x_r^+, x_1^-, \ldots, x_s^-)$ of $E(B)$ such that
$
f\left(
a^+_1, \dots, a^+_{k_1},\,
b^+_1, \dots, b^+_{r - k_1},\,
a^-_1, \dots, a^-_{k_2},\,
b^-_1, \dots, b^-_{s - k_2}
\right) \neq 0,
$
for some
$a^+_1 \in E(B_{i_1})^+, \dots, a^+_{k_1} \in E(B_{i_{k_1}})^+,
a^-_1 \in E(B_{i_{k_1 + 1}})^-, \dots, a^-_{k_2} \in E(B_{i_{k_1+k_2}})^-,
 b^+_1, \dots, b^+_{r - k_1} \in E(B)^+,$
 $
b^-_1, \dots, b^-_{s - k_2} \in E(B)^-$,
with $ r+s = k. $

It was proved in \cite{GPV2017} that
$exp^\ast(E(B))$
is the maximal dimension of an admissible subalgebra of $B$.
Analogously
$exp^{\ast,\delta}(E(B))$
is the maximal dimension of a centrally admissible subalgebra of $B$ (see \cite{MR2022}).

\section{Algebras with involution of small $\ast$-exponent}

In this section we will introduce some suitable algebras with involution  that will allow us to prove the main result of this paper.

Recall
that, if $F$ is an algebraically closed field of characteristic
zero, then, up to isomorphisms, all finite dimensional $\ast$-simple
are the following ones  (see \cite{Ro}, \cite{GZbook}):

\begin{itemize}
  \item $(M_n(F),t)$, the algebra of $n \times n$
matrices with the transpose involution;
  \item $(M_{2m}(F),s)$,  the algebra of $ 2m \times 2m
$ matrices with the symplectic involution;
  \item $(M_n(F)\oplus M_n(F)^{op},exc)$, the  direct
sum of the algebra of $n \times n$ matrices and the opposite
algebra with the exchange involution.
\end{itemize}

Now, we denote by $\theta_n$  an
important class of involutions on $M_{n}(F)$ of transpose type defined by
$$
e_{ij}^{\theta_n} = e_{n+1-j\,n+1-i} \,\,\, \text{for all} \,
i,j \in [1,n].
$$
The involution $\theta_n$  is called the reflection involution.
When $n = 2m$ is
an even positive integer we denote by $\sigma_{n}$ the symplectic type
involution on $M_n(F)$ defined by (see \cite{BDVV2024}):
$$e_{ij}^{\sigma_n} = \delta_{[1,m]}(i)\delta_{[1,m]}(j)e_{ij}^{\theta_n} = \delta_{[1,m]}(i)\delta_{[1,m]}(j)e_{n+1-j\,n+1-i}, \,\, \text{for all} \,
i,j \in [1,2m]$$
 where
$$\delta_{[1,m]}(i) = \left\{
\begin{array}{ll}
    1 & \hbox{if $i \in [1,m]$} \\
   -1 & \hbox{otherwise.}
\end{array}
\right.
$$

Here we recall  the following result:

\begin{proposition}\cite[Proposition 4]{BDVV2024} Let $F$ be an algebraically closed field of characteristic zero, and let $\ast$ be an involution on $M_n(F)$.
Then either $\ast$ is of transpose type and
$$
(M_n(F), \ast) \cong (M_n(F), t) \cong (M_n(F), \theta_n)
$$
or $n$ is even, $\ast$ is of symplectic type and
$$
(M_n(F), \ast) \cong (M_n(F), \sigma_n).
$$
\end{proposition}

\medskip

Let $UT_n(F)$ denote the subalgebra of $M_n(F)$ consisting of all
upper triangular matrices.

\medskip
\noindent Now we consider the following algebras:

\bigskip

\noindent
\begin{enumerate}
    \item[1)]
$N$,  the subalgebra of  $UT_6(F)$ of elements
$\begin{pmatrix}
       a &d & e & g & h & i  \\
        0 & b & f & 0 & 0& l \\
        0 & 0 & c & 0& 0& m\\
        0 & 0 & 0 & c & n & p\\
        0 & 0 & 0 & 0 & b & q\\
        0 & 0 & 0 & 0& 0 & a
    \end{pmatrix} $;

\bigskip
\bigskip
\noindent
\item[2)] $M$,  the subalgebra of  $UT_8(F)$ of elements
$\begin{pmatrix}
        0 & a & b & c & l & m & n & p \\
        0 &d &e & f & 0 & 0 & 0 & q  \\
        0 &0 & g & h &0 & 0 & 0 &r\\
        0 & 0 & 0 & i & 0 & 0 & 0 &s \\
        0 & 0 & 0 & 0&  i &t & u &v \\
        0 & 0 & 0 & 0&  0 & g & w &y \\
        0 & 0 & 0 & 0&  0 & 0 & d &z \\
        0 & 0 & 0 & 0&  0 & 0 & 0 &0
    \end{pmatrix}$;

\bigskip
\bigskip

\noindent
\item[3)] $P$, the subalgebra of  $UT_4(F)$ of elements
$\begin{pmatrix}
        a & d & e & f  \\
        0 & b &0 & g \\
        0 & 0& c & h \\
        0 & 0& 0 & a
    \end{pmatrix}$;

\bigskip
\bigskip

\noindent
\item[4)] $Q$, the subalgebra of  $UT_4(F)$ of elements
$\begin{pmatrix}
        a & d & e & f  \\
        0 & b &0 & g \\
        0 & 0& b & h \\
        0 & 0 &  0  & c
    \end{pmatrix}$;

\bigskip
\bigskip

\noindent
\item[5)] $R$, the subalgebra of  $UT_6(F)$ of elements
$\begin{pmatrix}
       0 &d & e & g & h & i  \\
        0 & a & f & 0 & 0& l \\
        0 & 0 & b & 0& 0& m\\
        0 & 0 &0 & c & n & p\\
        0 & 0 & 0 & 0 & a & q\\
        0 & 0 & 0 & 0& 0 & 0
    \end{pmatrix}.$
\end{enumerate}

\bigskip \bigskip

The following algebras with involution will be the focus of our considerations:
\bigskip

    \begin{enumerate}
                \item[ ]$\mathcal{A}_1=(M_2(F), t)$;

        \medskip
              \item[ ] $\mathcal{A}_2=(M_2(F), s)$;
        \medskip
         \item[ ] $\mathcal{A}_3=(M_{1,1}(E),\diamond )$; where
        $\begin{pmatrix}
        a & b \\
        c &d
    \end{pmatrix}^{\diamond}  = \begin{pmatrix}  d & b \\ -c &a
    \end{pmatrix}$;
\medskip
\item[ ] $\mathcal{A}_4=(E \oplus E^{op}, \textrm{exc})$;
\medskip
        \item[ ] $\mathcal{A}_5=(N,\theta_6)$;
        \medskip
        \item[ ] $\mathcal{A}_6= (N,\sigma_6)$;
        \medskip
        \item[ ] $ \mathcal{A}_7= (M,\theta_8)$;
        \medskip
        \item[ ] $\mathcal{A}_8= (M,\sigma_8)$;
        \medskip
        \item[ ]$\mathcal{A}_{9} = (P,\theta_4)$;
        \medskip
        \item[ ] $\mathcal{A}_{10}=(P,\sigma_4)$;
        \medskip
        \item[ ] $\mathcal{A}_{11}=(Q,\theta_4)$;
       \medskip
        \item[ ] $\mathcal{A}_{12}= (Q,\sigma_4)$;
        \medskip
        \item[ ] $\mathcal{A}_{13} =(R,\theta_6)$;
        \medskip
        \item[ ] $\mathcal{A}_{14} =(R,\sigma_6)$.
    \end{enumerate}

\bigskip

About their $*$-exponent and proper central $*$-exponent we have the following lemmas.

\begin{lemma} \label{exponent A1-A4}
		For $i=1,\ldots , 4$, $exp^\ast(\mathcal A_i)=exp^{\ast, \delta}(\mathcal A_i)=4$.
\end{lemma}
\begin{proof}

For the $\ast$-algebra $\mathcal{A}_1=(M_2(F), t)$ it is easy to see that the polynomial $f(x_1^-,x_2^-)=x_1^-x_2^-$ is a proper central $\ast$-polynomial, then $exp^\ast(\mathcal A_1)=exp^{\ast, \delta}(\mathcal A_1)=4.$

Since $f(x_1^+)=x_1^+$ is a proper central $\ast$-polynomial of $\mathcal{A}_2=(M_2(F), s),$ then,
also for his algebra, $exp^\ast(\mathcal A_2)=exp^{\ast, \delta}(\mathcal A_2)=4.$

For the $\ast$-algebra $\mathcal{A}_3= (M_{1,1}(E),\diamond )$, the polynomial $f(x_1^+,x_1^-,x_2^-)=[x_1^-,x_2^-,x_1^+]$ is a proper central $\ast$-polynomial, hence $exp^\ast(\mathcal A_3)=exp^{\ast, \delta}(\mathcal A_3)=4.$

Finally, it is easy to see  that $f(x_1^+,x_2^+)=[x_1^+,x_2^+]$ is a proper central $\ast$-polynomial for the $\ast$-algebra $\mathcal{A}_4=(E \oplus E^{op}, \textrm{exc})$ and so $exp^\ast(\mathcal A_4)=exp^{\ast, \delta}(\mathcal A_4)=4$
 and the proof of the lemma is completed.
\end{proof}

	\begin{lemma} \label{exponent A5-A14}
		For $i=5,\ldots , 14$, $exp^\ast(\mathcal A_i)=exp^{\ast, \delta}(\mathcal A_i)=3$.
	\end{lemma}
	
	\begin{proof}

For the $\ast$-algebra $\mathcal A_5$ we have
$
F(e_{11}+e_{66})Fe_{12}F(e_{22}+e_{55})Fe_{23}F(e_{33}+e_{44})\neq 0,$
so $F(e_{11}+e_{66})\oplus F(e_{22}+e_{55})\oplus F(e_{33}+e_{44})$ is a maximal admissible subalgebra. Hence $exp^\ast(\mathcal A_5)=3$.
The center of $\mathcal A_5$ is
$
Z(\mathcal A_5)=F(e_{11}+\cdots+e_{66})+Fe_{16}.
$

Since
$
[e_{11}+e_{66},e_{12}+e_{56}][e_{22}+e_{55},e_{23}+e_{45}][e_{33}+e_{44},e_{36}+e_{14}]=e_{16},
$
the polynomial
$
[x_1^+,x_2^+][x_3^+,x_4^+][x_5^+,x_6^+]
$
is a proper central $\ast$-polynomial for $\mathcal A_5$. It follows that
$F(e_{11}+e_{66})\oplus F(e_{22}+e_{55})\oplus F(e_{33}+e_{44})$ is a maximal centrally admissible subalgebra and $exp^{\ast,\delta}(\mathcal A_5)=3$.

The same arguments apply to $\mathcal A_6$, so
$
exp^\ast(\mathcal A_6)=exp^{\ast,\delta}(\mathcal A_6)=3.
$

For both  $\mathcal A_7$ and $\mathcal A_8$ a maximal admissible subalgebra is
$
F(e_{22}+e_{77})\oplus F(e_{33}+e_{66})\oplus F(e_{44}+e_{55}),
$
hence $exp^\ast(\mathcal A_7)=exp^\ast(\mathcal A_8)=3$.
Moreover their centers coincide,
$
Z(\mathcal A_7)=Z(\mathcal A_8)=Fe_{18}.
$

For $\mathcal A_7$ the polynomial
$
[x_1^+,x_2^+][x_3^+,x_4^+][x_5^+,x_6^+][x_7^+,x_8^+]
$
is a proper central $\ast$-polynomial
since $[e_{12}+e_{78},e_{22}+e_{77}][e_{22}+e_{77},e_{23}+e_{67}][e_{33}+e_{66},e_{34}+e_{56}][e_{44}+e_{55}, e_{48}+ e_{14}]= e_{18} \in Z(\mathcal{A}_7)$; therefore $exp^{\ast,\delta}(\mathcal A_7)=3$.

For $\mathcal A_8$ the polynomial
$
[x_1^+,x_2^+][x_3^+,x_4^+][x_5^+,x_6^+][x_7^+,x_8^-]
$
is a proper central $\ast$-polynomial in fact
$[e_{12}+e_{78},e_{22}+e_{77}][e_{22}+e_{77},e_{23}+e_{67}][e_{33}+e_{66},e_{34}+e_{56}][e_{44}+e_{55}, e_{48}- e_{14}]= - e_{18} \in Z(\mathcal{A}_8)$ and so $exp^{\ast,\delta}(\mathcal A_8)=3$.

A maximal admissible subalgebra for $\mathcal A_9$ and $\mathcal A_{10}$ is
$
F(e_{11}+e_{44})\oplus (F\oplus F)(e_{22}+e_{33}),
$
then $exp^\ast(\mathcal A_9)=exp^\ast(\mathcal A_{10})=3$.
Notice that
$
F(e_{11}+e_{22}+e_{33}+e_{44})+Fe_{14}=Z(\mathcal A_9)=Z(\mathcal A_{10}).
$
The polynomial $[x_1^+,x_2^+][x_4^+,x_5^+]$ is a proper central $\ast$-polynomial for $\mathcal A_9$, and $[x_1^+,x_2^+][x_4^+,x_5^-]$ is a proper central $\ast$-polynomial for $\mathcal A_{10}$. Hence,
$
exp^{\ast,\delta}(\mathcal A_9)=exp^{\ast,\delta}(\mathcal A_{10})=3.
$

The same arguments apply to $\mathcal A_{11}$ and $\mathcal A_{12}$, giving
$
exp^\ast(\mathcal A_{11})=exp^\ast(\mathcal A_{12})=3$ and $exp^{\ast,\delta}(\mathcal A_{11}) = exp^{\ast,\delta}(\mathcal A_{12})=3.
$

Finally, a maximal admissible subalgebra for both  $\mathcal A_{13}$ and $\mathcal A_{14}$
is
$
F(e_{22}+e_{33})\oplus (F\oplus F)(e_{44}+e_{55}).
$
Moreover $Z(\mathcal A_{13})=Z(\mathcal A_{14}) =Fe_{16}$. The polynomial
$
[x_1^+,x_2^+][x_3^+,x_4^+][x_5^+,x_6^+]
$
(resp.\ $[x_1^+,x_2^+][x_3^+,x_4^+][x_5^+,x_6^-]$) is a proper central $\ast$-polynomial for $\mathcal A_{13}$ (resp.\ $\mathcal A_{14}$). Therefore
$
exp^\ast(\mathcal A_{13})=exp^\ast(\mathcal A_{14})=3$ and
$
exp^{\ast,\delta}(\mathcal A_{13})=exp^{\ast,\delta}(\mathcal A_{14})=3,
$
which completes the proof.
\end{proof}

\bigskip
 \section{The main results}
    The aim of this section is to investigate varieties of $\ast$-algebras whose proper central $\ast$-exponent exceeds two. To achieve this, we first establish several auxiliary lemmas.

	\begin{lemma} \label{lemmaA_5,A_6}
    Let $A=A^{+} \oplus A^{-}$ be a finite-dimensional $\ast$-algebra. \black
		If $A$ contains three orthogonal symmetric idempotents $e_1,e_2,e_3$ such that
		$$e_1j_1e_2j_2e_3j_3e_1\ne 0,$$
        for some $j_1,j_2,j_3\in A$, then  either $Id^\ast(A)\subseteq Id^\ast(\mathcal{A}_5)$ or $Id^\ast(A)\subseteq Id^\ast(\mathcal{A}_6)$.
	\end{lemma}
\begin{proof} Without loss of generality we may assume that the elements $j_1, j_2,j_3\in A^{+} \cup A^{-}.$
		Let $B$ be the $\ast$-subalgebra of $A$ generated by the elements
		$$
		e_1,  e_2,  e_3, e_1j_1e_2, e_2j_2e_3, e_3j_3e_1.
		$$
First suppose that the elements $e_1j_1e_2j_2e_3j_3e_1$ and $(e_1j_1e_2j_2e_3j_3e_1)^\ast$ are linearly dependent over $F$ and suppose $(e_1j_1e_2j_2e_3j_3e_1)^\ast=\alpha e_1j_1e_2j_2e_3j_3e_1$, for some $\alpha \in F.$ Since the involution $\ast$ has order two, we have that $\alpha= \pm 1.$

Let $I$ be the $\ast$-ideal of $B$ generated  by the elements
		$$
		e_3j_3e_1j_1e_2, e_1j_1e_2j_1^\ast e_1, e_2j_1^\ast e_1j_1e_2, e_2j_2e_3j_2^\ast e_2, e_3j_2^\ast e_2j_2e_3,
        e_3j_3 e_1j_3^\ast e_3,  e_1j_3^\ast e_3j_3 e_1.
		$$
Then the $\ast$-algebra $B/I$ is finite dimensional and the following elements
$$
e_1,  e_2,  e_3, e_1j_1e_2, e_2j_1^\ast e_1, e_2j_2e_3, e_3j_2^\ast e_2, e_3j_3e_1, e_1j_3^\ast e_3,
$$
$$
e_1j_1e_2j_2e_3, e_3j_2^\ast e_2j_1^\ast e_1,
e_2j_2e_3j_3e_1, e_1j_3^\ast e_3j_2^\ast e_2,
e_1j_1e_2j_2e_3j_3e_1
$$
are  linearly independent and form a basis of $B$ mod $I$.
By abuse of notation we identify these representatives with the corresponding cosets.

We build a linear map $\varphi : B/I \rightarrow N$  by setting
$$\varphi(e_1)=e_{11}+e_{66}, \varphi(e_2)=e_{22}+e_{55}, \varphi(e_3)=e_{33}+e_{44},$$
$$\varphi(e_1j_1e_2)=e_{12}, \varphi(e_2j_1^\ast e_1)=e_{56},
\varphi(e_2j_2e_3)=e_{23}, \varphi(e_3j_2^\ast e_2)=e_{45},
$$
$$
\varphi(e_3j_3e_1)=e_{36}, \varphi(e_1j_3^\ast e_3)=\alpha e_{14},
\varphi(e_1j_1e_2j_2e_3)=e_{13}, \varphi(e_3j_2^\ast e_2j_1^\ast e_1)=e_{46},
$$
$$
\varphi(e_2j_2e_3j_3e_1)=e_{26}, \varphi(e_1j_3^\ast e_3j_2^\ast e_2)=\alpha e_{15},
\varphi(e_1j_1e_2j_2e_3j_3e_1)=e_{16}.
$$

This map $\varphi$ extends to an isomorphism of $\ast$-algebras
		and $B/I\simeq\mathcal{A}_5$ in case $\alpha=1$ and  $B/I\simeq\mathcal{A}_6$ in case $\alpha=-1$.
It follows that either  $Id^\ast(\mathcal{A}_5) \supseteq  Id^\ast(A)$ or $Id^\ast(\mathcal{A}_6) \supseteq  Id^\ast(A)$.

Now, we suppose that the elements $e_1j_1e_2j_2e_3j_3e_1$ and $(e_1j_1e_2j_2e_3j_3e_1)^\ast$ are linearly independent over $F$.
In this case if  we add respectively
\[
e_1 j_1 e_2 j_2 e_3 j_3 e_1 - (e_1 j_1 e_2 j_2 e_3 j_3 e_1)^\ast
\quad \text{or} \quad
e_1 j_1 e_2 j_2 e_3 j_3 e_1 + (e_1 j_1 e_2 j_2 e_3 j_3 e_1)^\ast
\]
to the generators of $I$; in the first case we obtain $B/I\simeq\mathcal{A}_5$, by choosing $\alpha =1$, and in the second case $B/I\simeq\mathcal{A}_6$ by choosing $\alpha =-1$.
\end{proof}

\bigskip

\begin{lemma} \label{lemmaA_7,A_8}
  Let $A=A^{+} \oplus A^{-}$ be a finite-dimensional $\ast$-algebra.
		If $A$ contains three orthogonal symmetric idempotents $e_1,e_2,e_3$ such that
		$$j_1e_1j_2e_2j_3e_3j_4\ne 0,$$
        for some $j_1,j_2,j_3, j_4\in A$, then  either $Id^\ast(A)\subseteq Id^\ast(\mathcal{A}_7)$ or $Id^\ast(A)\subseteq Id^\ast(\mathcal{A}_8)$.
	\end{lemma}
\begin{proof}
As in the previous lemma, we assume that the elements   $j_1, j_2,j_3\in A^{+} \cup A^{-}.$
Let  $B$ be the $\ast$-subalgebra of $A$ generated by the elements
		$$
		e_1,  e_2,  e_3, j_1e_1, e_1j_2e_2, e_2j_3e_3, e_3j_4.
		$$
Suppose first  that the elements $j_1e_1j_2e_2j_3e_3j_4$ and $(j_1e_1j_2e_2j_3e_3j_4)^\ast$ are linearly dependent over $F$, then $(j_1e_1j_2e_2j_3e_3j_4)^\ast=\alpha j_1e_1j_2e_2j_3e_3j_4$, with $\alpha= \pm 1.$

Now, we consider $I$ the $\ast$-ideal of $B$
		generated by the elements
		$$
		e_3j_4e_1, e_3j_4e_2, e_3j_4e_3, e_3j_4j_1e_1, e_3j_1e_1, e_2j_1e_1, e_1j_1e_1, j_1e_1j_1^\ast,
        j_4^\ast e_3 j_4,
        $$
        $$
        e_1j_1^\ast j_1e_1, e_3j_4 j_4^\ast e_3, e_1j_2e_2j_2^\ast e_1,
        e_2j_2^\ast e_1j_2e_2, e_2j_3e_3j_3^\ast e_2,
        e_3j_3^\ast e_2j_3 e_3.	$$

\noindent
Then $B/I$ is finite dimensional and the elements
$$
e_1,  e_2,  e_3, j_1e_1,  e_1j_1^\ast , e_1j_2e_2, e_2j_2^\ast e_1, e_2j_3e_3, e_3j_3^\ast e_2, e_3j_4, j_4^\ast e_3,
$$
$$
j_1e_1j_2e_2, e_2j_2^\ast e_1j_1^\ast,
e_1j_2e_2j_3e_3, e_3j_3^\ast e_2j_2^\ast e_1,
e_2j_3e_3j_4, j_4^\ast e_3j_3^\ast e_2,
$$
$$
j_1e_1j_2e_2j_3e_3, e_3j_3^\ast e_2j_2^\ast e_1 j_1^\ast,
e_1j_2e_2j_3e_3j_4, j_4^\ast e_3j_3^\ast e_2 j_2^\ast e_1,
j_1e_1j_2e_2j_3e_3j_4
$$
are  linearly independent and span $B$ mod $I$.
By abuse of notation we identify these representatives with the corresponding cosets.

We build a linear map $\psi : B/I \rightarrow M$  by setting
$$
\psi(e_1)=e_{22}+e_{77}, \psi(e_2)=e_{33}+e_{66}, \psi(e_3)=e_{44}+e_{55},
$$
$$
\psi(j_1e_1)=e_{12}, \psi(e_1j_1^\ast)=e_{78},
\psi(e_1j_2e_2)=e_{23}, \psi(e_2j_2^\ast e_1)=e_{67},
$$
$$
\psi(e_2j_3e_3)=e_{34}, \psi(e_3j_3^\ast e_2)=e_{56},
\psi(e_3j_4)=e_{48}, \psi(j_4^\ast e_3)=\alpha e_{15},
$$
$$
\psi(j_1e_1j_2e_2)=e_{13}, \psi(e_2j_2^\ast e_1j_1^\ast)=e_{68},
\psi(e_1j_2e_2j_3e_3)=e_{24}, \psi(e_3j_3^\ast e_2j_2^\ast e_1)=e_{57},
$$
$$
\psi(e_2j_3e_3j_4)=e_{38},\psi(j_4^\ast e_3j_3^\ast e_2)=\alpha e_{16},
\psi(j_1e_1j_2e_2j_3e_3)=e_{14}, \psi(e_3j_3^\ast e_2j_2^\ast e_1j_1^\ast)=e_{58},
$$
$$
\psi(e_1j_2e_2j_3e_3j_4)=e_{28}, \psi(j_4^\ast e_3j_3^\ast e_2j_2^\ast e_1)=\alpha e_{17},
\psi(j_1e_1j_2e_2j_3e_3j_4)= e_{18}.
$$

\smallskip
This map extends to an isomorphism of $\ast$-algebras, yielding $B/I\simeq\mathcal{A}_7$ if $\alpha=1$ or  $B/I\simeq\mathcal{A}_8$ if $\alpha=-1$.
Hence either  $Id^\ast(\mathcal{A}_7) \supseteq  Id^\ast(A)$ or $Id^\ast(\mathcal{A}_8) \supseteq  Id^\ast(A)$.

\smallskip

If instead $j_1 e_1 j_2 e_2 j_3 e_3 j_4$ and $(j_1e_1j_2e_2j_3e_3j_4)^\ast$ are linearly independent, then by adding  $e_1j_1e_2j_2e_3j_3e_1-(e_1j_1e_2j_2e_3j_3e_1)^\ast$ or $e_1j_1e_2j_2e_3j_3e_1+(e_1j_1e_2j_2e_3j_3e_1)^\ast$ to the generators of $I$, we obtain
$
B/I \simeq
\mathcal{A}_7$ {if } $\alpha=1,$ or
$ B/I \simeq \mathcal{A}_8$ \text{if } $\alpha=-1$ respectively.

This completes the proof.
\end{proof}	

\bigskip

\begin{lemma} \label{lemmaA_9,A_{10}}
 Let $A=A^{+} \oplus A^{-}$ be a finite-dimensional $\ast$-algebra.
		If $A$ contains two orthogonal symmetric idempotents $e_1,e_2$, with $e_1 \in F$ and $e_2\in (F\oplus F, exc)$, such that
		$$
        e_1 j_1 e_2 j_2 e_1 \ne 0
        $$
        for some $j_1,j_2\in A$, then  either $Id^\ast(A)\subseteq Id^\ast(\mathcal{A}_9)$ or $Id^\ast(A)\subseteq Id^\ast(\mathcal{A}_{10})$.
	\end{lemma}
\begin{proof}
Let $B$ be the $\ast$-subalgebra of $A$ generated by the elements
		$$
		e_1,  e_2, e_2^-, e_1j_1e_2, e_2j_2e_1,
		$$
where $e_2^-= (1,-1)e_2.$
First suppose that the elements $e_1j_1e_2 j_2e_1$ and $(e_1j_1e_2 j_2e_1)^\ast$ are linearly dependent over $F$ and let $(e_1j_1e_2 j_2e_1)^\ast=\alpha e_1j_1e_2 j_2e_1$, for some $\alpha = \pm 1.$
Now, we consider $I$ the $\ast$-ideal of $B$
		generated by the elements
		$$
		e_2j_2e_1j_1e_2, e_2j_1^\ast e_1j_1e_2,
        e_1j_2^\ast e_2j_2e_1, e_1j_1e_2j_1^\ast e_1,
        e_2j_2 e_1j_2^\ast e_2,
         e_1j_1e_2-e_1j_1e_2^-,
        e_2j_2e_1-e_2^-j_2e_1.
		$$

\noindent
The $\ast$-algebra $B/I$ is finite dimensional and the following elements
$$
e_1,  e_2,  e_2^-,  e_1j_1e_2, e_2j_1^\ast e_1, e_2j_2e_1, e_1j_2^\ast e_2,
e_1j_1e_2j_2e_1
$$
are  linearly independent and form a basis of $B$ mod $I$.
As in the previous lemmas, we identify these representatives with their cosets.

Define a linear map $\theta : B/I \rightarrow P$  by setting

$$
\theta(e_1)=e_{11}+e_{44}, \theta(e_2)=e_{22}+e_{33}, \theta(e_2^-)=e_{22}-e_{33}, \theta(e_1j_1e_2)=e_{12},
$$
$$
 \theta(e_2j_1^\ast e_1)=e_{34},
\theta(e_2j_2e_1)=e_{24}, \theta(e_1j_2^\ast e_2)=\alpha e_{13},
\theta(e_1j_1e_2j_2e_1)=e_{14}.
$$

\bigskip

The map $\theta$ extends to an isomorphism of $\ast$-algebras
		and $B/I\simeq\mathcal{A}_9$ in case $\alpha=1$ and  $B/I\simeq\mathcal{A}_{10}$ in case $\alpha=-1$.
It follows that either  $Id^\ast(\mathcal{A}_9) \supseteq  Id^\ast(A)$ or $Id^\ast(\mathcal{A}_{10}) \supseteq  Id^\ast(A)$.

Now, we suppose that the elements $e_1j_1e_2 j_2e_1$ and $(e_1j_1e_2 j_2e_1)^\ast$ are linearly independent over $F$. In this case  we can add
\[
e_1 j_1 e_2 j_2 e_1 - (e_1 j_1 e_2 j_2 e_1)^\ast
\quad \text{or} \quad
e_1 j_1 e_2 j_2 e_1 + (e_1 j_1 e_2 j_2 e_1)^\ast
\]
to the generators of $I$, yielding
$B/I \simeq \mathcal{A}_9,$ by choosing  $\alpha =1,$ or $B/I \simeq\mathcal{A}_{10}$, by choosing $\alpha =-1,$ respectively.

The proof is now completed.
\end{proof}	

\begin{lemma} \label{lemmaA_{11},A_{12}}
 Let $A=A^{+} \oplus A^{-}$ be a finite-dimensional $\ast$-algebra.
		If $A$ contains two orthogonal symmetric idempotents $e_1,e_2$, with $e_1 \in F$ and $e_2\in (F\oplus F, exc)$, such that
		$$
        e_2 j_1 e_1 j_2 e_2 \ne 0
        $$
        for some $j_1,j_2\in A$, then  either $Id^\ast(A)\subseteq Id^\ast(\mathcal{A}_{11})$ or $Id^\ast(A)\subseteq Id^\ast(\mathcal{A}_{12})$.
	\end{lemma}
\begin{proof}
Let $B$ be the $\ast$-subalgebra of $A$ generated by the elements
		$$
		e_1,  e_2, e_2^-, e_2j_1e_1, e_1j_2e_2,
		$$
where $e_2^-= (1,-1)e_2.$
Suppose that $(e_2 j_1 e_1 j_2 e_2)^\ast=\alpha e_2 j_1 e_1 j_2 e_2$, for some $\alpha = \pm 1.$
Let $I$  be the $\ast$-ideal of $B$
		generated by the elements
		$$
		e_1j_2e_2j_1e_1, e_1j_2e_2j_2^\ast e_1,
        e_1j_1^\ast e_2j_1e_1, e_2j_1e_1j_1^\ast e_2,
        e_2j_2^\ast e_1j_2 e_2,
         e_1j_2e_2+e_1j_2e_2^-,
        e_2j_1e_1-e_2^-j_1e_1.
		$$

\noindent
The $\ast$-algebra $B/I$ is finite dimensional and the following elements
$$
e_1,  e_2,  e_2^-,  e_2j_1e_1, e_1j_1^\ast e_2, e_1j_2e_2, e_2j_2^\ast e_1,
e_2j_1e_1j_2e_2
$$
are  linearly independent and form a basis of $B$ mod $I$.
We identify these representatives with their cosets in $B/I$.
Now, we define a linear map $\zeta : B/I \rightarrow Q$  by setting

$$
\zeta(e_1)=e_{22}+e_{33}, \zeta(e_2)=e_{11}+e_{44}, \zeta(e_2^-)=e_{11}-e_{44}, \zeta(e_2j_1e_1)=e_{13},
$$
$$
 \zeta(e_1j_1^\ast e_2)=\alpha e_{24},
\zeta(e_1j_2e_2)=e_{34}, \zeta(e_2j_2^\ast e_1)= e_{12},
\zeta(e_2j_1e_1j_2e_2)= e_{14}.
$$

\bigskip

The map $\zeta$ extends to an isomorphism of $\ast$-algebras
		and $B/I\simeq\mathcal{A}_{11}$ in case $\alpha=1$ and  $B/I\simeq\mathcal{A}_{12}$ in case $\alpha=-1$.
It follows that either  $Id^\ast(\mathcal{A}_{11}) \supseteq  Id^\ast(A)$ or $Id^\ast(\mathcal{A}_{12}) \supseteq  Id^\ast(A)$.

Now, we suppose that the elements $e_2 j_1 e_1 j_2 e_2$ and $(e_2 j_1 e_1 j_2 e_2)^\ast$ are linearly independent over $F$.
In this case adding either
\[
e_2 j_1 e_1 j_2 e_2 - (e_2 j_1 e_1 j_2 e_2)^\ast
\quad \text{or} \quad
e_2 j_1 e_1 j_2 e_2 + (e_2 j_1 e_1 j_2 e_2)^\ast
\]
to the generators  of $I$ we get $B/I \simeq \mathcal{A}_{11}$ by choosing $\alpha =1$, or $\mathcal{A}_{12}$, by choosing $\alpha =-1,$ respectively.

The proof is now completed.
\end{proof}	

\bigskip

\begin{lemma} \label{lemmaA_{13},A_{14}}
  Let $A=A^{+} \oplus A^{-}$ be a finite-dimensional $\ast$-algebra.
		If $A$ contains two orthogonal symmetric idempotents $e_1,e_2$, with $e_1 \in F$ and $e_2\in (F\oplus F, exc)$, such that
		$$
        j_1e_1 j_2 e_2 j_3 \ne 0 \,\, \mathrm{or} \,\, j_1e_2 j_2 e_1 j_3 \ne 0
        $$
        for some $j_1,j_2,j_3\in A$, then  either $Id^\ast(A)\subseteq Id^\ast(\mathcal{A}_{13})$ or $Id^\ast(A)\subseteq Id^\ast(\mathcal{A}_{14})$.
	\end{lemma}
\begin{proof}
First let $j_1e_1 j_2 e_2 j_3 \ne 0$. We consider $B$ the $\ast$-subalgebra of $A$ generated by the elements
		$$
		e_1,  e_2, e_2^-, j_1e_1,  e_1j_2e_2, e_2j_3,
		$$
where $e_2^-= (1,-1)e_2.$
Suppose that $(j_1e_1 j_2 e_2 j_3 )^\ast=\alpha j_1e_1 j_2 e_2 j_3 $, with $\alpha = \pm 1,$ and  let $I$ be the $\ast$-ideal of $B$ generated by the elements
	$$
		e_1j_1e_1, e_2j_1e_1, e_2j_3e_1, e_2j_3e_2,e_2j_3j_1e_1,
        e_1j_1^\ast j_1e_1, j_1e_1j_1^\ast
        $$
        $$
        e_2j_2^\ast e_1j_2 e_2,  e_1j_2e_2j_2^\ast  e_1,
     j_3^\ast e_2j_3,  e_2j_3j_3^\ast e_2,
        e_1j_2e_2-e_1j_2e_2^-,
        e_2j_3-e_2^-j_3.
		$$
    \black
\noindent
Then $B/I$ is finite-dimensional  and the following elements
$$
e_1,  e_2,  e_2^-,  j_1e_1, e_1j_2e_2, e_2j_3, j_1e_1j_2e_2, e_1j_2e_2j_3,$$
$$
j_1e_1j_2e_2j_3, e_1j_1^\ast, e_2j_2^\ast e_1, j_3^\ast e_2, e_2j_2^\ast e_1 j_1^\ast, j_3^\ast e_2 j_2^\ast e_1
$$
are  linearly independent and form a basis of $B$ mod $I$.
As in the previous lemmas we identify these representatives with the corresponding cosets.
Now, we consider a linear map $\xi: B/I \rightarrow R$  by setting
$$
\xi(e_1)=e_{22}+e_{55}, \xi(e_2)=e_{33}+e_{44}, \xi(e_2^-)=e_{33}-e_{44},  \xi(j_1e_1)=e_{12}, \xi(e_1j_2e_2)=e_{23},
$$
$$
\xi(e_2j_3)=e_{36}, \xi(j_1e_1j_2e_2)=e_{13}, \xi(e_1j_2e_2j_3)=e_{26},
\xi(j_1e_1j_2e_2j_3)= e_{16},
$$
$$
 \xi(e_1j_1^\ast)= e_{56},
 \xi(e_2j_2^\ast e_1)=e_{45},
\xi(j_3^\ast e_2)=\alpha e_{14}, \xi(e_2j_2^\ast e_1j_1^\ast)= e_{46},
\xi(j_3^\ast e_2j_2^\ast e_1)=\alpha e_{15}.
$$

\bigskip

The map $\xi$ extends to an isomorphism of $\ast$-algebras
		and $B/I\simeq\mathcal{A}_{13}$ in case $\alpha=1$ and  $B/I\simeq\mathcal{A}_{14}$ in case $\alpha=-1$.
It follows that either  $Id^\ast(\mathcal{A}_{13}) \supseteq  Id^\ast(A)$ or $Id^\ast(\mathcal{A}_{14}) \supseteq  Id^\ast(A)$.

Now, we suppose that the elements $j_1e_1 j_2 e_2 j_3$ and $(j_1e_1 j_2 e_2 j_3)^\ast$ are linearly independent over $F$. In this case if we add the element $j_1e_1 j_2 e_2 j_3-(j_1e_1 j_2 e_2 j_3)^\ast$ to the generators of $I$ then we get $B/I\simeq\mathcal{A}_{13}$ by choosing $\alpha =1$. Instead if we add the element $j_1e_1 j_2 e_2 j_3+(j_1e_1 j_2 e_2 j_3)^\ast$ to the generators of $I$ then we get $B/I\simeq\mathcal{A}_{14}$ by choosing $\alpha =-1$.

Let $j_1e_2 j_2 e_1 j_3 \ne 0$. We consider $B$ the $\ast$-subalgebra of $A$ generated by the elements
		$$
		e_1,  e_2, e_2^-, j_1e_2,  e_2j_2e_1, e_1j_3.
		$$
Suppose first  that $(j_1e_2 j_2 e_1 j_3  )^\ast=\alpha j_1e_2 j_2 e_1 j_3$, with $\alpha = \pm 1,$ and  let $I$ be the $\ast$-ideal of $B$ generated by the elements
	$$
		e_1j_1e_2, e_2j_1e_2, e_1j_3e_1, e_1j_3e_2, e_1j_3j_1e_2,
        e_2j_1^\ast j_1e_2, j_1e_2j_1^\ast
        $$
        $$
        e_1j_2^\ast e_2j_2 e_1,  e_2j_2e_1j_2^\ast  e_2,
     j_3^\ast e_1j_3,  e_1j_3j_3^\ast e_1,
        e_2j_2e_1+e_2^-j_2e_1,
        j_1e_2+j_1e_2^-.
		$$
    \black
\noindent
The $\ast$-algebra $B/I$ is spanned by the following elements
$$
e_1,  e_2,  e_2^-,  j_1e_2, e_2j_2e_1, e_1j_3, j_1e_2j_2e_1, e_2j_2e_1j_3,$$
$$
j_1e_2j_2e_1j_3, e_2j_1^\ast, e_1j_2^\ast e_2, j_3^\ast e_1, e_1j_2^\ast e_2 j_1^\ast, j_3^\ast e_1 j_2^\ast e_2.
$$
As in the previous lemmas we identify these representatives with the corresponding cosets.
Now, we consider a linear map $\chi: B/I \rightarrow R$  by setting
$$
\chi(e_1)=e_{22}+e_{55}, \chi(e_2)=e_{33}+e_{44}, \chi(e_2^-)=e_{33}-e_{44},  \chi(j_1e_2)=e_{14}, \chi(e_2j_2e_1)=e_{45},
$$
$$
\chi(e_1j_3)=e_{56}, \chi(j_1e_2j_2e_1)=e_{15}, \chi(e_2j_2e_1j_3)=e_{46},
\chi(j_1e_2j_2e_1j_3)=\ e_{16},
$$
$$
 \chi(e_2j_1^\ast)= \alpha e_{36},
 \chi(e_1j_2^\ast e_2)=e_{23},
\chi(j_3^\ast e_1)=e_{12}, \chi(e_1j_2^\ast e_2j_1^\ast)= \alpha e_{26},
\chi(j_3^\ast e_1j_2^\ast e_2)= e_{13}.
$$

\bigskip

The map $\chi$ extends to an isomorphism of $\ast$-algebras
		and $B/I\simeq\mathcal{A}_{13}$ in case $\alpha=1$ and  $B/I\simeq\mathcal{A}_{14}$ in case $\alpha=-1$.
It follows that either  $Id^\ast(\mathcal{A}_{13}) \supseteq  Id^\ast(A)$ or $Id^\ast(\mathcal{A}_{14}) \supseteq  Id^\ast(A)$.

Now, we suppose that the elements $j_1e_2j_2e_1j_3$ and $(j_1e_2j_2e_1j_3)^\ast$ are linearly independent over $F$. In this case if we add the element $j_1e_2j_2e_1j_3-(j_1e_2j_2e_1j_3)^\ast$ to the generators of $I$ then we get $B/I\simeq\mathcal{A}_{13}$ by choosing $\alpha =1$. Instead if we add the element $j_1e_2j_2e_1j_3+(j_1e_2j_2e_1j_3)^\ast$ to the generators of $I$ then we get $B/I\simeq\mathcal{A}_{14}$ by choosing $\alpha =-1$.

The proof is now completed.
\end{proof}

   Now we are in a position to state the main result of this section.

    \begin{theorem} \label{teorema1}
        Let $F$ be a field of characteristic zero and $\mathcal{V}^\ast$ a variety of  $\ast$-algebras over $F$.
        Then  $exp^{\ast,\delta}(\mathcal V^\ast)> 2$ if  $\mathcal{A_i}$ belongs to $\mathcal V^\ast$  for some $ i= 1, \ldots, 14.$
    \end{theorem}

    \begin{proof}

    Let $exp^{\ast,\delta}(\mathcal V^\ast)> 2$. We may assume that $\mathcal V^\ast = var^\ast(A)= var^\ast(E(B))$ where $E(B)=(E^{(0)}\otimes B^{(0)}) \oplus (E^{(1)}\otimes B^{(1)})$ is the Grassmann envelope of a finite-dimensional superalgebra with superinvolution $B=B^{(0)}\oplus B^{(1)}.$
        Write $B=B_1\oplus \cdots \oplus B_q +J$, where the $B_i$'s are simple superalgebras with superinvolution and let  $exp^{\ast,\delta}(\mathcal{V^\ast})=exp^{\ast,\delta}(A)= exp^{\ast,\delta}(E(B))> 2$.

      Suppose first that some $B_i\cong M_{k,2h}(F)$ with orthosymplectic superinvolution. If $h\ge1$, then
\[
E(B)\supseteq E(B_i)\cong E(M_{k,2h}(F)) \supseteq E^{(0)}\otimes
\begin{pmatrix} 0 & 0\\ 0 & M_{2h}(F)\end{pmatrix}
\]
and hence contains
\[
1\otimes \big(Fe_{k+1,k+1} \oplus Fe_{k+1,k+h+1} \oplus Fe_{k+h+1,k+1}\oplus Fe_{k+h+1,k+h+1}\big)
\cong (M_2(F),s).
\]
Therefore $Id^\ast(E(B))\subseteq Id^\ast((M_2(F),s))$, so $A_2= (M_2(F),s)\in\mathcal V^\ast$.

      \smallskip   If $h =0$ and $k \ge 2,$ then  $B_i \cong ( M_{k}(F), t)$ and
         $$
         E(B)\supseteq E(B_i) \cong E^{(0)} \otimes  M_{k}(F) \supseteq E^{(0)} \otimes  M_{2}(F).
         $$
      Hence $Id^\ast(E(B))\subseteq Id^\ast((M_2(F), t))$  and so $A_1=(M_2(F), t) \in \mathcal{V}^\ast.$

Hence we may assume that whenever $B_i\cong (M_{k,2h}(F),osp)$ we have $h=0$ and $k=1$, i.e. $B_i\cong F$ with the trivial superinvolution.

Next suppose for some $i$ that $B_i\cong M_{k,k}(F)$ with transpose superinvolution. If $k\ge2$ then
\[
E(B)\supseteq E(B_i)\supseteq E^{(0)}\otimes \big( \begin{pmatrix}M_k(F)&0\\0&M_k(F)\end{pmatrix},trp\big)
\supseteq E^{(0)}\otimes (H,t)\cong (M_k(F),t)\supseteq (M_2(F),t),
\]
where $H= \left\lbrace  \begin{pmatrix}
       A & 0   \\
       	0 & A
       \end{pmatrix} \,| \,A \in M_{k}(F) \right\rbrace $. Thus $A_2=(M_2(F),t)\in\mathcal V^\ast. $

       If $k=1$ then
$E(B)\supseteq E(B_i)\cong E((M_{1,1}(F),trp))=(M_{1,1}(E),\diamond)$ and hence $A_3\in \mathcal{V}^\ast$.
 \smallskip

Now suppose some $B_i\cong (M_{k,l}(F)\oplus M_{k,l}(F)^{sop},exc)$.

If $ k \ge 2,$  let  $K= \left\lbrace (U,U^t) \,| \, U \in M_{k}(F) \right\rbrace \subseteq (M_{k}(F) \oplus M_{k}(F)^{op}, exc)$
then
$$
E(B)\supseteq E((M_{k,l}(F) \oplus M_{k,l}(F)^{sop}, exc))\supseteq  (M_{k}(F) \oplus M_{k}(F)^{op}, exc))\supseteq
$$
$$
E^{(0)} \otimes  (K, exc) \cong (M_{k}(F),t) \supseteq (M_{2}(F),t).
$$

\noindent It follows that $Id^\ast(E(B)) \subseteq Id^\ast(E(B_i)) \subseteq Id^\ast((M_{2}(F),t))$  and so $A_1 \in \mathcal{V}^\ast.$

\smallskip

\noindent If $k =l=1,$ then we may consider $(M_{1,1}(F) \oplus M_{1,1}(F), \ast)$ with superinvolution defined by $(U,V)^\ast = (V^{trp},U^{trp}).$ Hence the map

$$ \psi: (M_{1,1}(F) \oplus M_{1,1}(F)^{sop},exc) \to (M_{1,1}(F) \oplus M_{1,1}(F), \ast)$$  defined by $ \psi ((U,V)) = (U, V^{trp})$ is an isomorphism of superalgebras with superinvolution. If we denote by $Q =\left\lbrace (A, A) \,| \, A \in M_{1,1}(F)\right\rbrace $ the $\ast$-subalgebra of $ (M_{1,1}(F) \oplus M_{1,1}(F), \ast)$ then
 $$
E(B)\supseteq E(B_i)\cong E((M_{1,1}(F) \oplus M_{1,1}(F)^{sop},exc))\cong E((M_{1,1}(F) \oplus M_{1,1}(F),\ast)) \supseteq $$
$$E((Q,\ast)) \cong E((M_{1,1}(F), trp))= (M_{1,1}(E),\diamond).
$$

\noindent Hence  if $B_i\cong (M_{k,l}(F) \oplus M_{k,l}(F)^{sop}, exc),$ we may assume that $k=1, l=0$ and so $B_i$ is isomorphic to the $\ast$-algebra $(F\oplus F,exc).$

 \medskip

Now let
$B_i\cong (M_{k}(F\oplus cF) \oplus M_{k}(F\oplus cF)^{sop}, exc).$
Then $$M_{k}(F\oplus cF)\oplus M_{k}(F\oplus cF)^{sop} \supseteq M_k(F) \oplus M_k(F)^{op} $$  and so
$E(M_{k}(F\oplus cF)\oplus M_{k}(F\oplus cF)^{sop}) \supseteq E^{(0)} \otimes (M_k(F) \oplus M_k(F)^{op}).$

If $k \ge 2,$ as in the previous case we obtain that $A_1 = (M_2(F),t) \in \mathcal{V}^\ast.$

If $k =1,$ then $B_i \cong (F+cF) \oplus (F+cF)^{sop} $ with exchange superinvolution. In this case we get that
$E(B)\supseteq E(B_i) \cong (E \oplus E^{op}, exc)$ and so
$Id^\ast(E(B)) \subseteq Id^\ast(E(B_i)) = Id^\ast((E \oplus E^{op}, exc)).$  Hence $A_4 \in \mathcal V^\ast.$

\medskip

       As an outcome of the above discussion we may assume that for every simple component $B_i$ we have either $B_i\cong F$ with trivial superinvolution or $B_i\cong F\oplus F$ with exchange involution. Let $D=B_1\oplus\cdots\oplus B_r$ be a centrally admissible $\ast$-subalgebra of $B$ of maximal dimension $d$. By definition there is a multilinear proper central $\ast$-polynomial
\[
f=f(x_1^+,\dots,x_u^+,x_1^-,\dots,x_v^-)
\]
of $E(B)$ with $u+v\ge r$ having a nonzero evaluation involving symmetric and skew-symmetric elements from each $E(B_i)$ for $i=1,\dots,r$.

Since $exp^{\ast,\delta}(E(B))=\dim D=d\ge3$, one of the following occurs:
\begin{enumerate}
  \item $D$ contains at least three simple components of type $F$;
  \item $D$ contains exactly one copy of $F$ and one copy of $F\oplus F$;
  \item $D$ contains only two copies of $F\oplus F$.
\end{enumerate}

        Let us assume first that Case 1 occurs. $D$ contains  three orthogonal symmetric idempotents $e_1,e_2,e_3$.
        Since $f$ is the  proper central $\ast$-polynomial corresponding to $D$ we have that in a non-zero evaluation $\bar f$ of $f$  at least three variables are evaluated in
        $$
        a_1\otimes e_1,\  a_2\otimes e_2,\  a_3\otimes e_3,
        $$
        where $a_1,a_2,a_3\in E^{(0)}$ are distinct elements.
        Since $E^{(0)}$ is central in $E$ and $f$ is multilinear, we may assume that $a_1=a_2=a_3=1$
        and, so,  the elements $1\otimes e_i,\  1\le i \le 3$ appear in the evaluation of $f$.

        Suppose first that a monomial of $f$ can be evaluated into a product of the type
        $$
        (1\otimes e_1)(b_1\otimes j_1)(1\otimes e_2)(b_2\otimes j_2) (1\otimes e_3)(b_3\otimes j_3)(1\otimes e_1)\ne 0,
        $$
        for some  $b_1\otimes j_1, b_2\otimes j_2, b_3\otimes j_3  \in E^{(0)}\oplus J^{(0)} \cup  E^{(1)}\oplus J^{(1)}$.
Then $e_1j_1e_2j_2e_3j_3e_1\ne 0$ and, by Lemma \ref{lemmaA_5,A_6}, we get that either $Id^\ast(A_5) \supseteq Id^\ast(A)$ or $Id^\ast(A_6) \supseteq Id^\ast(A)$ and we are done in this case. This clearly holds for any permutation of the $e_i$'s, $ i=1,2,3$.

     Hence if $m$ is a monomial of $f$ whose evaluation $\bar m$ in $ \bar f$ is non zero, we may assume  that  $(1\otimes e_i)\bar m(1\otimes e_i)=0$, $1\le i\le 3$. This also says that  $(1\otimes e_i)\bar f(1\otimes e_i)=0$  and since $\bar f$ is central   $(1\otimes e_i)\bar f = \bar f(1\otimes e_i)=0.$

        Next suppose that there exists an evaluation of a monomial of $f$ containing a product of the type
        $$
         (b_1\otimes j_1)(1\otimes e_1) (b_2\otimes j_2)(1\otimes e_2)
         (b_3\otimes j_3)(1\otimes e_3) (b_4\otimes j_4)\ne 0,
        $$
        for some   $b_1\otimes j_1, b_2\otimes j_2, b_3\otimes j_3, b_4\otimes j_4  \in E^{(0)}\oplus J^{(0)} \cup  E^{(1)}\oplus J^{(1)}$. Then $j_1e_1j_2e_2j_3e_3j_4\ne 0$.
        By Lemma \ref{lemmaA_7,A_8}, we get that either   $Id^\ast(A_7) \supseteq Id^\ast(A)$ or  $Id^\ast(A_8)  \supseteq Id^\ast(A)$  and we are done.

        Hence we may assume that in  $\bar f$ an evaluation of a monomial never starts and ends with elements of $J.$

        Let us set $\bar e_i=1\otimes e_i$, $1\le i\le 3$, and write
        $$
        \bar f = \bar e_1\bar f_1+ \bar f_2 \bar e_1+ \bar e_2 \bar f_3 + \bar f_4 \bar e_2+ \bar e_3\bar f_5+ \bar f_6 \bar e_3,
        $$

        \noindent where $\bar e_i\bar f_{2i-1}$  is the sum of the evaluations of the monomials of $f$ starting  with $\bar e_i$, and
        $\bar f_{2i} \bar e_i$ is the  sum of the evaluations of the remaining monomials of $f$ ending with $\bar e_i$, with $i=1,2,3$.

        By the first part of the proof we may assume that  $\bar e_i\bar f_j \bar e_i=0$ for $1 \le i \le 3$ and $ 1 \le j \le 6
        $. Since $\bar e_1\bar f=\bar f \bar e_2=\bar f \bar e_3=0$,  we get
        $$
        0=\bar e_1\bar f \bar e_2=\bar e_1\bar f_1 \bar e_2 + \bar e_1\bar f_4 \bar e_2,
        $$
        $$
        0=\bar e_1\bar f \bar e_3=\bar e_1\bar f_1 \bar e_3 + \bar e_1\bar f_6 \bar e_3.
        $$
        Moreover, from
        $$
        \bar e_1\bar f+\bar f \bar e_2+\bar f \bar e_3=0
        $$

        \noindent  and by using the previous equalities, we obtain that

        $$
        \bar e_1\bar f_1 + \bar f_4 \bar e_2 + \bar f_6 \bar e_3=-\bar e_3 \bar f_5 \bar e_2 - \bar e_2 \bar f_3 \bar e_3.
        $$

        \bigskip
        \noindent Similarly,
        since $\bar f \bar e_1=\bar e_2 \bar f =\bar e_3 \bar f =0$ we get
        $$
        \bar f_2 \bar e_1 + \bar e_2 \bar f_3 +\bar e_3 \bar f_5 =-\bar e_2 \bar f_6 \bar e_3 - \bar e_3 \bar f_4 \bar e_2.
        $$
        In conclusion we have that
        $$
        \bar f= \bar e_1 \bar f_1+ \bar f_2 \bar e_1  +  \bar e_2 \bar f_3+ \bar f_4 \bar e_2+ \bar e_3 \bar f_5 + \bar f_6 \bar e_3
        =-\bar e_3 \bar f_5 \bar e_2 - \bar e_2 \bar f_3 \bar e_3-\bar e_2 \bar f_6 \bar e_3 - \bar e_3 \bar f_4 \bar e_2.
        $$
        Now, since
        $
        0=\bar f \bar e_2$ it follows that$-\bar e_3\bar f_5 \bar e_2 - \bar e_3 \bar f_4 \bar e_2=0.$ Then
        $\bar f=  - \bar e_2 \bar f_3 \bar e_3-\bar e_2 \bar f_6 \bar e_3$  and we reach
        $
        \bar f = \bar f \bar e_3=0,$
        a contradiction.
        This completes Case 1.

\medskip
        Now, we consider the second case. Then $D $ contains
        one copy of $F$ and one of $F \oplus F$. Let $e_1, e_2$ be orthogonal symmetric idempotents of $F$ and of $F \oplus F$ respectively.
         We remark that $ F \oplus F = (F \oplus F)^+ \oplus (F \oplus F)^- = Fe_2 \oplus F(1,-1)e_2= Fe_2 \oplus Fe_2^-.$
         It is clear that the multilinear proper central $\ast$-polynomial $f$ has a non-zero evaluation $\bar f$ in which
        at least one variable is evaluated in $1\otimes e_1$
        and one variable  in  $1\otimes e_2$ or  $1\otimes e_2^-.$

       Suppose first that there exists an evaluation of a monomial of $f$ such that
        $$
        (1\otimes e_1) (b_1\otimes j_1)(1\otimes e_2) (b_2\otimes j_2)(1\otimes e_1) \ne 0,
        $$
        for some $b_1\otimes j_1, b_2\otimes j_2 \in E^{(0)}\oplus J^{(0)} \cup  E^{(1)}\oplus J^{(1)}.$
        Then, $e_1j_1e_2j_2e_1\ne 0$ and,
        by Lemma \ref{lemmaA_9,A_{10}}, we obtain that  either
        $Id^\ast(A_9) \supseteq Id^\ast(A)$ or $Id^\ast(A_{10})\supseteq Id^\ast(A).$

        If there exists an evaluation of a monomial of $f$ such that
        $$
        (1\otimes e_1) (b_1\otimes j_1)(1\otimes e_2^-) (b_2\otimes j_2)(1\otimes e_1) \ne 0,
        $$
        for some $b_1\otimes j_1, b_2\otimes j_2 \in E^{(0)}\oplus J^{(0)} \cup  E^{(1)}\oplus J^{(1)},$ then $e_1 j_1 (1,-1)e_2 j_2 e_1 \ne 0.$ If we put $j_1'=j_1(1,-1)$ we are done by the  previous case.

        Next let us consider the case when there exists an evaluation of a monomial of $f$ such that
        $$
        (1\otimes  e_2)(b_1\otimes j_1)(1\otimes e_1) (b_2\otimes j_2)(1\otimes  e_2)\ne 0,
        $$
        for some $b_1\otimes j_1, b_2\otimes j_2 \in E^{(0)}\oplus J^{(0)} \cup  E^{(1)}\oplus J^{(1)}.$
By Lemma \ref{lemmaA_{11},A_{12}},  we get that $Id^\ast(A_{11})\supseteq Id^\ast(A)$ or $Id^\ast(A_{12})\supseteq Id^\ast(A).$
If there exists an evaluation of a monomial of $f$ such that
 $$(1\otimes  e_2^-)(b_1\otimes j_1)(1\otimes e_1) (b_2\otimes j_2)(1\otimes  e_2)\ne 0$$
 or
 $$(1\otimes  e_2)(b_1\otimes j_1)(1\otimes e_1) (b_2\otimes j_2)(1\otimes  e_2^-)$$
 or
 $$(1\otimes  e_2^-)(b_1\otimes j_1)(1\otimes e_1) (b_2\otimes j_2)(1\otimes  e_2^-)\ne 0,
 $$
for some $b_1\otimes j_1, b_2\otimes j_2 \in E^{(0)}\oplus J^{(0)} \cup  E^{(1)}\oplus J^{(1)},$ then  $e_2j_1e_1j_2e_2\ne 0$ and we are done by the previous case.

Now assume that a monomial of $f$ has an evaluation of the type
        $$
         (b_1\otimes j_1)(1\otimes e_1) (b_2\otimes j_2) (1\otimes e_2)
         (b_3\otimes j_3)\ne 0,
        $$
        for some $b_1\otimes j_1, b_2\otimes j_2, b_3\otimes j_3  \in E^{(0)}\oplus J^{(0)} \cup  E^{(1)}\oplus J^{(1)}$.
        Let $j_1e_1j_2e_2j_3 \ne 0.$
         By Lemma \ref{lemmaA_{13},A_{14}},
        it follows that $Id^\ast(A_{13})\supseteq Id^\ast(A)$ or $Id^\ast({A}_{14})\supseteq Id^\ast(A)$. The same result holds when $j_1e_1j_2e_2^-j_3 \ne 0.$

       Finally we consider the case when a monomial of $f$ has an evaluation of the type
        $$
         (b_1\otimes j_1)(1\otimes e_2)(b_2\otimes j_2)(1\otimes e_1)
         (b_3\otimes j_3)\ne 0,
        $$
        for some $b_1\otimes j_1, b_2\otimes j_2, b_3\otimes j_3  \in E^{(0)}\oplus J^{(0)} \cup  E^{(1)}\oplus J^{(1)}$.
        By Lemma \ref{lemmaA_{13},A_{14}}, we have that  $Id^\ast(A_{13})\supseteq Id^\ast(A)$ or $Id^\ast({A}_{14})\supseteq Id^\ast(A)$. The same result holds when $j_1e_2^-j_2e_1j_3 \ne 0.$

        As in the first case if $\bar m\ne 0$ is the evaluation in $\bar f$ of a monomial of $f$, we may assume that
        $(1\otimes e_1)\bar m (1\otimes e_1)=0$ and  $(1\otimes e_2)\bar m (1\otimes e_2)=0$. Then, since $\bar f$ is central,
        $(1 \otimes e_i) \bar f = \bar f (1 \otimes e_i) =0$,  $i=1,2$.
        If we denote $(1 \otimes e_i)=\bar e_i$, $i=1,2$, then  we may write, as in the previous case,
        $$
        \bar f = \bar e_1\bar f_1+ \bar f_2 \bar e_1+ \bar e_2 \bar f_3 + \bar f_4 \bar e_2.
        $$
        Since $\bar e_1\bar f=0$ we obtain that
        $\bar e_1\bar f_1  + \bar e_1\bar f_4 \bar e_2= 0$
        and so
        $\bar e_1\bar f_1 \bar e_2 + \bar e_1\bar f_4 \bar e_2= 0.$
        It follows that $\bar e_1\Bar{f}_1 = \bar e_1\Bar{f}_1\bar e_2.$
        Moreover, since $\bar f\bar e_2 =0$,  similarly we get
        $
        \bar e_1\bar f_1\bar e_2+ \bar f_4 \bar e_2= 0
        $
        and so
        $
        \bar e_1\bar f_1+\bar f_4 \bar e_2= 0.
        $
        As a consequence we obtain that
        $$
        \bar f =  \bar f_2 \bar e_1+ \bar  e_2 \bar f_3.
        $$
        Since
        $0= \bar e_2\bar f =  \bar e_2\bar f_2 \bar e_1+ \bar e_2 \bar f_3,$
        we get $\bar f =  \bar f_2 \bar e_1+ \bar e_2 \bar f_3= \bar f_2 \bar e_1 - \bar e_2\bar f_2 \bar e_1$ and so
        $0=\bar f \bar e_1= \bar f_2 \bar e_1 - \bar e_2\bar f_2 \bar e_1 = \bar f \neq 0$,  a contradiction.
        This completes Case 2.

       We remark that, if $D$ contains two copies of $F\oplus F$, then it also contains one copy of $F$ and one copy of $F\oplus F$.
        This says that Case 3 can be deduced from Case 2.
    \end{proof}

    \bigskip

	Next we recall the following
	
	\begin{definition}
		Let $\mathcal{V}^\ast$ be a variety of $\ast$-algebras. We say that $\mathcal{V}^\ast$ is minimal of proper central $\ast$-exponent $d\geq 2$ if $exp^{\ast,\delta}(\mathcal{V}^\ast)=d$ and for
		every proper subvariety $\mathcal{U}^\ast\subset \mathcal{V}^\ast$ we have that $exp^{\ast,\delta}(\mathcal{U}^\ast)<d$.
	\end{definition}
	
	\medskip

Let $\mathcal{V}_i^\ast=var^\ast(\mathcal{A}_i)$, for $1\le i\le 14$.
The next goal is to prove that $\mathcal{V}_i^\ast$, for $1\le i\le 14$, are the only minimal $\ast$-varieties of proper central $\ast$-exponent equal to $3$ or $4.$

 We begin by showing that these  $\ast$-varieties are not comparable, i.e., if $\mathcal{V}^\ast$ and $\mathcal{W}^\ast$ are any two such $\ast$-varieties then $\mathcal{V}^\ast\not \subseteq  \mathcal{W}^\ast$ and $\mathcal{W}^\ast\not \subseteq  \mathcal{V}^\ast$.

	\medskip

	\begin{proposition} \label{differentvarieties}
		For  $1\le i,j\le 14$, $i\neq j$, we have that  $\mathcal{V}^\ast_i \not \subseteq \mathcal{V}^\ast_j$.
	\end{proposition}
    \begin{proof}
    We start by comparing among themselves varieties of the same proper central $\ast$-exponent.

 By an easy check of the following  statements we can say that $\mathcal{V}^\ast_i \not \subseteq \mathcal{V}^\ast_j$, for  $1\le i,j\le 4$, $i\neq j$.
		
		\begin{itemize}
			\item[-]
			
			$[x_1^-x_2^-,x_3^-] \in Id^\ast(\mathcal{A}_1)$ but
			$\not \in Id^\ast(\mathcal{A}_2)$ and $\not \in Id^\ast(\mathcal{A}_3)$,
			\smallskip

			\item[-]
			
			$[x_1^+,x_2^+]\in Id^\ast(\mathcal{A}_2)$ and $\in Id^\ast(\mathcal{A}_3)$ but $\not \in  Id^\ast(\mathcal{A}_1)$ and $\not \in  Id^\ast(\mathcal{A}_4)$,
			\smallskip
			
			\item[-]
			$[x_1^+,x_1^-]\in Id^\ast(\mathcal{A}_2)$ but
			$ \not \in Id^\ast(\mathcal{A}_3)$,
			
			\smallskip
			
			\item[-]
    			$[x_1^-,x_2^-][x_3^-,x_4^-] \in Id^\ast(\mathcal{A}_3)$ but $  \not \in   Id^\ast(\mathcal{A}_2)$,
			
			\smallskip
			
			\item[-] $[x_1^-,x_2^-,x_3^-] \in Id^\ast(\mathcal{A}_4)$ but $\not \in Id^\ast(\mathcal{A}_2)$ and $\not \in Id^\ast(\mathcal{A}_3)$,

            \smallskip
			
			\item[-] $[x_1^-x_2^-,x_1^+] \in Id^\ast(\mathcal{A}_1)$ but $\not \in Id^\ast(\mathcal{A}_4)$,
            \smallskip
			
			\item[-] $[x_1^+,x_2^+,x_3^+] \in Id^\ast(\mathcal{A}_4)$ but $\not \in Id^\ast(\mathcal{A}_1)$.
		\end{itemize}

Now, consider the varieties $\mathcal{V}^\ast_i$ for $5 \le i,j \le 14$, $i \neq j$.
By the following  statements we have that $\mathcal{V}^\ast_i \not \subseteq \mathcal{V}^\ast_j$, for  $5\le i,j\le 14$, $i\neq j$.
		\begin{itemize}
			\item[-]
			
			$[x_1^-x_2^-,x_3^-] \in Id^\ast(\mathcal{A}_5)$ but
			$\not \in Id^\ast(\mathcal{A}_i)$ for $6\le i\le 8$,
			\smallskip

			\item[-]
			
			$[x_1^-,x_2^- ]\circ x_3^-\in Id^\ast(\mathcal{A}_6)$  but $\not \in  Id^\ast(\mathcal{A}_5)$ and $\not \in  Id^\ast(\mathcal{A}_i)$, with $i=7, 8$,
			\smallskip
			
			\item[-]
			$x_1^+[x_2^+,x_3^+][x_4^+,x_5^+][x_6^+,x_7^+]x_8^+\in Id^\ast(\mathcal{A}_i)$, with $i=7, 8$, but
			$ \not \in Id^\ast(\mathcal{A}_j)$, for $j=5, 6$,
			
			\smallskip
			
			\item[-]
    			$[x_1^-,x_2^-][x_3^-,x_4^-] \in Id^\ast(\mathcal{A}_i),$ for $j=5, 6$, but $  \not \in   Id^\ast(\mathcal{A}_j)$, with $9\le j\le 14$,
			
			\smallskip
			
			\item[-]  $[x_1^+,x_2^+][x_3^+, x_4^+][x_5^+, x_6^+] \in Id^\ast(\mathcal{A}_i)$ for $9\le i \le 12$ but $\not \in Id^\ast(\mathcal{A}_j)$  with $5\le j \le 8$ or $j=13,14$,

            \item[-]  $[x_1^+,x_2^+][x_3^+, x_4^+][x_5^+, x_6^+]x_7^+\in Id^\ast(\mathcal{A}_i)$ for $ i =13,14$ but $\not \in Id^\ast(\mathcal{A}_j)$  with $5\le j \le 8$,

            \smallskip
			
			\item[-] $[x_1^-x_2^-,x_3^-, x_4^-] \in Id^\ast(\mathcal{A}_7)$ but $\not \in Id^\ast(\mathcal{A}_8)$,

            \smallskip
			
			\item[-] $[x_1^-,x_2^-,x_3^-] \circ x_4^- \in Id^\ast(\mathcal{A}_8)$ but $\not \in Id^\ast(\mathcal{A}_7)$,

            \smallskip
			
			\item[-] $[x_1^-x_2^-][x_3^-, x_4^-] s_1^+\in Id^\ast(\mathcal{A}_i)$ with $i=7,8$ but $\not \in Id^\ast(\mathcal{A}_j)$, for $9\le j \le 14$,

            \smallskip

            \item[-] $[[x_1^-x_2^-][x_3^-, x_4^-]]\in Id^\ast(\mathcal{A}_i)$ with $i=9,11$ but $\not \in Id^\ast(\mathcal{A}_j)$, for $j =10, 12$,

            \smallskip

             \item[-] $[x_1^-x_2^-]\circ [x_3^-, x_4^-]\in Id^\ast(\mathcal{A}_i)$ with $i=10, 12$ but $\not \in Id^\ast(\mathcal{A}_j)$, for $j =9,11$,

            \smallskip

             \item[-] $[x_1^-x_2^-] [x_3^-, x_4^-]x_5^-\in Id^\ast(\mathcal{A}_i)$ with $i=9,10$ but $\not \in Id^\ast(\mathcal{A}_j)$, for $11\le j \le 14$,

            \smallskip

            \item[-]  $[x_1^-x_2^-] x_3^- [x_4^-, x_5^-]\in Id^\ast(\mathcal{A}_i)$ with $i=11,12$ but $\not \in Id^\ast(\mathcal{A}_j)$, for $j=9,10,13, 14$,\black

            \smallskip

             \item[-] $[[[x_1^-x_2^-],  [x_3^-, x_4^-]], x_5^-]\in Id^\ast(\mathcal{A}_{13})$ but $\not \in Id^\ast(\mathcal{A}_{14})$,

            \smallskip

             \item[-] $[[x_1^-x_2^-],  [x_3^-, x_4^-]] \circ x_5^-\in Id^\ast(\mathcal{A}_{14})$ but $\not \in Id^\ast(\mathcal{A}_{13})$,

            \smallskip
		\end{itemize}

        At this point we compare among themselves varieties of different proper central $\ast$-exponent. It is easy to see that $x_1^-x_2^-x_3^-x_4^-x_5^-\in  Id^\ast(\mathcal{A}_i)$, for $5\le i \le 14$ and $\not \in Id^\ast(\mathcal{A}_j)$ with $j=1,2,3,4$.

        Finally, the proof is completed by remarking that $[x_1^-x_2^-,x_3^-]\in Id^\ast(\mathcal{A}_{1})$ but $\not \in Id^\ast(\mathcal{A}_j)$, for $5\le j \le 14$,
        $[x_1^+x_2^+]\in Id^\ast(\mathcal{A}_{i})$ with $i=2,3$ but $\not \in Id^\ast(\mathcal{A}_j)$, for $5\le j \le 14$ and $[x_1^+x_2^+,x_3^+]\in Id^\ast(\mathcal{A}_{4})$ but $\not \in Id^\ast(\mathcal{A}_j)$, for $5\le j \le 14$.

\end{proof}

As a consequence, we get the following
	
	\begin{corollary}  A $\ast$-variety $\mathcal{V}^\ast$ is minimal  of proper central $\ast$-exponent equal to $4$ if and only if $\mathcal{V}^\ast=\mathcal{V}^\ast_i,$  for some  $i=1, \ldots ,4$.
		A $\ast$-variety $\mathcal{V}^\ast$ is minimal  of proper central $\ast$-exponent equal to $3$ if and only if  $\mathcal{V}^\ast=\mathcal{V}^\ast_i,$  for some  $i=5, \ldots , 14$.
        \end{corollary}

	\begin{proof}
    By Lemmas \ref{exponent A1-A4}, \ref{exponent A5-A14},
$exp^{\ast,\delta}(\mathcal{V}^\ast_i)=4$, for $1 \le i \le 4$ and $exp^\delta(\mathcal{V}^\ast_i)=3$, for $5\le i\le 14$. Now,
let $\mathcal{U}^\ast$ be a proper subvariety of $\mathcal{V}^\ast_i$, for some $i \in \{ 1, \ldots , 14\}$,
and suppose that $exp^{\ast,\delta}(\mathcal{U}^\ast)\ge$ $exp^{\ast,\delta}(\mathcal{V}^\ast_i)> 2$.	
		Then, by Theorem \ref{teorema1} and  Proposition \ref{differentvarieties},  we obtain  a contradiction.

 Now if $\mathcal{V}^\ast$ is a minimal $\ast$-variety  of proper central $\ast$-exponent equal to $3$ or $4$, then by Theorem \ref{teorema1} it must be coincide with $\mathcal{V}^\ast_i,$  for some $1\le i\le 14$.
	\end{proof}

\end{document}